\documentclass[12pt,reqno]{amsart}

\usepackage[utf8]{inputenc}

\title{Growth in sumsets of higher convex functions}
\author{Peter J. Bradshaw}

\usepackage[a4paper,top=3cm,bottom=2cm,left=3cm,right=3cm,marginparwidth=1.75cm]{geometry}

\usepackage{amsmath}
\usepackage{amssymb}
\usepackage{amsthm}
\usepackage{enumitem}
\usepackage{dsfont}
\usepackage{verbatim}
\usepackage{xcolor}

\theoremstyle{definition}
\newtheorem{define}{Definition}[section]

\theoremstyle{plain}
\newtheorem{thm}{Theorem}[section]
\newtheorem{lem}{Lemma}[section]
\newtheorem{prop}{Proposition}[section]

\newtheorem{conj}{Conjecture}[section]

\theoremstyle{remark}
\newtheorem*{rmk}{Remark}

\newtheorem*{notation}{Notation}

\newcommand{\N}{\mathbb{N}}
\newcommand{\Q}{\mathbb{Q}}

\newcommand{\R}{\mathbb{R}}
\newcommand{\C}{\mathbb{C}}

\newcommand{\F}{\mathbb{F}}

\usepackage[myheadings]{fullpage}
\usepackage{fancyhdr}
\usepackage{lastpage}
\usepackage{graphicx, wrapfig, setspace}
\usepackage[T1]{fontenc}
\usepackage[english]{babel}

\begin{document}

\maketitle

\begin{abstract}
    The main results of this paper concern growth in sums of a $k$-convex function $f$.
    Firstly, we streamline the proof (from \cite{higher_convexity}) of a growth result for $f(A)$ where $A$ has small additive doubling, and improve the bound by removing logarithmic factors.
    The result yields an optimal bound for 
    \[ |2^k f(A) - (2^k-1)f(A)|. \]
    We also generalise a recent result from \cite{hansonRocheNewtonSenger2021}, proving that for any finite $A\subset \R$
    \[ | 2^k f(sA-sA) - (2^k-1) f(sA-sA)| \gg_s |A|^{2s}  \]
     where $s = \frac{k+1}{2}$.
    This allows us to prove that, given any natural number $s \in \N$, there exists $m = m(s)$ such that if $A \subset \R$, then 
    \begin{equation}\label{conj A-Aus}
    |(sA-sA)^{(m)}| \gg_s |A|^{s}.    
    \end{equation}
    This is progress towards a conjecture \cite{BalogRoche-NewtonZhelezov} which states that \eqref{conj A-Aus} can be replaced with 
    \[|(A-A)^{(m)}| \gg_s |A|^{s}.\]
    
    Developing methods of Solymosi, and Bloom and Jones, and using an idea from \cite{BHRConvexity}, we present some new sum-product type results in the complex numbers $\C$ and in the function field $\F_q((t^{-1}))$. 
\end{abstract}

\section{Introduction}
Studying growth under sums and differences of convex sets has received significant attention over the last twenty years. 
Recently, there has been important progress culminating in, among other results, a threshold-breaking bound for the number of dot products in the plane \cite{hansonRocheNewtonSenger2021}.
Another central problem which appears related is whether enough products of the difference set $A-A$ must grow ad infinitum. 

\subsection{Growth of sets under convex functions}
Let $A$ be a finite set of reals.
The sum-product problem is one of the most studied problems in additive combinatorics. 
It deals with the relative sizes of the sumset and product set:
\begin{equation}\label{sumsets productsets}
A+A = \{a+a': a,a'\in A\} \qquad \text{and} \qquad AA = \{aa': a,a'\in A\}. 
\end{equation}

The following was conjectured by Erd\"os and Szemer\'edi \cite{ErdosSzemeredi1983} (originally over the integers).
\begin{conj} [Sum-Product Conjecture]
For any positive $\delta <1$, there exists a constant $C(\delta)>0$, such that for any sufficiently large $A\subset \R$, we have 
\begin{equation}\label{sum-prod}
   \max\{|A+A|,|AA|\} \geq  C(\delta) |A|^{1+\delta}. 
\end{equation}
\end{conj}
\begin{notation}
We use Vinogradov's symbol extensively. We write $X \ll Y$ to mean that $X \leq CY$ for some absolute constant $C$, and $X \lesssim Y$ to mean that $X \leq C_1Y(\log Y)^{C_2}$ for some absolute constants $C_1,C_2$.
If $X \ll Y$ and $Y \ll X$, we will write $X \approx Y$.
If additionally a suppressed constant $C:=C(k)$ depends on some parameter $k$, we write $\ll_k, \lesssim_k, \approx_k$. Usually in such cases $k$ will be fixed and therefore small compared to the growing parameters. 

Equipped with this notation,
we could rewrite \eqref{sum-prod} as
\[ \max\{ |A+A|,|AA| \} \gg_\delta |A|^{1+\delta}. \]

\end{notation}

While the full sum-product conjecture remains open, proving the above for increasing values of $\delta$ represents the progress of the last forty years. The best known to date is $\delta = 1/3 + \frac{2}{1167} - o(1)$ due to Rudnev and Stevens \cite{rudnevStevens2021}.

The notation in \eqref{sumsets productsets} naturally generalises to incorporate differences and ratios as well as longer sums. One of the key quantities of interest in this paper will be
\[ A+A-A = \{a+a'-a'': a,a',a'' \in A \}. \]
For longer sums and products we will suppress notation so that $nA$ is the $n$-fold sumset and $A^{(n)}$ is the $n$-fold product set. Incorporating differences and division, we allow sets of the form $mA-nA$ and $A^{(m)}/A^{(n)}$.

A set $B = \{ b_1 < \dots < b_N \}$ is said to be convex if
\[ b_2-b_1 < b_3-b_2 < \dots < b_N - b_{N-1}. \]
It can be easily shown that a convex set can be expressed as $A = f([N])$ where $f$ is an increasing function with an increasing first derivative and $[N]:=\{1,\dots,N\}$.  

We say a $C^1(\R)$ function $f$ is \emph{convex} if both $f$ and $f'$ are either strictly increasing or strictly decreasing. There is an  aphorism in additive combinatorics that convex functions destroy additive structure; if $f$ is a convex function, then either $A$ or $f(A)$ can be additively structured, but not both.
Statements about the additive structure of $A$ and $f(A)$ are more general than sum-product results.
Indeed, if for some positive $\delta <1$, we have 
\begin{equation}\label{f sp thm}
\max\{|A+A|,|f(A)+f(A)|\} \gg |A|^{1+\delta},
\end{equation}
then one immediately obtains a sum-product theorem by choosing $f(x) = \log(x)$. 
Elekes, Nathanson and Ruzsa \cite{ElekesNathansonRuzsa} proved a more general form of \eqref{f sp thm} with $\delta = 1/4$. 
We firstly prove the following, a generalisation of \eqref{f sp thm} to longer sums/differences.

\begin{thm} \label{maintheorem}
Let $A$ be a finite set of reals and $f$ be a convex function. Then
\[ |A+A-A| |f(A)+f(A)-f(A)| \gg |A|^3. \]
\end{thm}
In \cite{higher_convexity}, Hanson, Roche-Newton and Rudnev proved the slightly weaker statement
\[ |A+A-A||f(A)+f(A)-f(A)| \gg \frac{|A|^3}{\log^3 |A|}. \]
The removal of the logarithmic factors makes the bound of Theorem \ref{maintheorem} sharp.
Indeed, if $f(x) = x^2$ and $A = [N]$ then
\[ |A+A-A||f(A)+f(A)-f(A)| \approx |A|^3. \]

Theorem \ref{maintheorem} is simply the base case for the forthcoming Theorem \ref{removing logs theorem}. 
It is stated separately because we present a short, new proof inspired by the very simple method by which Solymosi establishes the $\delta = 1/4$ sum-product bound in $\C$ \cite{SolymosiComplexSP}.

A simple idea which emerged in a recent paper of Hanson, Rudnev and the author \cite{BHRConvexity} is that $A+A-A$ must, in some sense, be evenly spaced among the elements of $A$. More specifically, if $a,a'\in A$ have very few elements of $A+A-A$ between them, then $a$ and $a'$ must be close to each other. The proof of Theorem \ref{maintheorem} elucidates how this idea may be leveraged in proving sumset bounds.

\subsection{Higher convex functions}
It is expected that the more convex a function, the less likely $f(A)$ is to exhibit additive structure.
We say a $C^k(\R)$ function is \emph{$k$-convex} if it has strictly monotone derivatives $f^{(0)}, f^{(1)},\dots, f^{(k)}$. 
Here monotone means either increasing or decreasing.
A $1$-convex function is simply a convex function and a $0$-convex function is a strictly monotone function.
To streamline exposition, we henceforth assume that a $k$-convex function $f$ has monotone \emph{increasing} derivatives $f^{(0)}, f^{(1)}, \dots, f^{(k)}$.
If any of them are decreasing, the proofs herein can be easily modified to handle this. Such modifications are discussed in \cite{higher_convexity}.
It is nonetheless worth noting that if $f$ has all derivatives positive, then $F(x) := -f(-x)$ has the signs of its derivatives alternating and $F$ has the same sumset behaviour as $f$. 
This is a more direct way to see that our theorems will apply in the important case where $f = \log$.

The forthcoming Theorem \ref{removing logs theorem} improves the main result in \cite{higher_convexity} by logarithmic factors. Specifically, we use an idea from \cite{BHRConvexity} to sidestep the need for dyadic pigeonholing.

\begin{thm} \label{removing logs theorem}
Let $A$ be a finite set of real numbers and $f$ be a $k$-convex function. Then if $|A+A-A|\leq K|A|$, we have
\[ |2^kf(A)-(2^k-1)f(A)| \gg_k \frac{|A|^{2^{k+1}-1}}{|A+A-A|^{2^{k+1}-k-2}} = |A|^{k+1} K^{-(2^{k+1}-k-2)}. \]
\end{thm}
Setting
$A = [N]$ and $f(x) = x^{k+1}$
verifies that the term $|A|^{k+1}$ on the right-hand side the best possible. 

In \cite{hansonRocheNewtonSenger2021}, it is proved that if $A$ is a finite real set and $f$ is any convex function, then
\[ |2f(A\pm A) - f(A\pm A)| \gg |A|^2. \]
We prove the following generalisation.
\pagebreak
\begin{thm}\label{thm many A inside}
Let $A$ be a set of reals and $f$ be a $k$-convex function. 
If $k$ is even with $k = 2s$ then
\[ | 2^k f((s+1)A-sA) - (2^k-1) f((s+1)A-sA)| \gg_k |A|^{k+1} \]
and if $k$ is odd with $k = 2s-1$ then 
\[ | 2^k f(sA-sA) - (2^k-1) f(sA-sA)| \gg_k |A|^{k+1}. \]
\end{thm}

Notice that in both cases, the number of summands inside the function $f$ is the same as the power of $|A|$ on the right-hand side. 
For the purposes of discussion, assume we are in the odd $k$ case.
A generic set $A$ is expected to have $|sA-sA| = |A|^{2s} = |A|^{k+1}$, and therefore $|f(sA-sA)| = |A|^{k+1}$ as well since $f$ is a monotone function. 
However, obviously $sA-sA$ may be significantly smaller than $|A|^{k+1}$.
Theorem \ref{thm many A inside} affirms that even in this case, sufficiently many sums (and differences) of $f(sA-sA)$ guarantee the same $|A|^{k+1}$ growth. 

While it may be possible to improve Theorem \ref{thm many A inside} by reducing the number of summands of $f(sA-sA)$ on the left hand side, the bound $|A|^{k+1}$ cannot be improved, evinced by setting
$A = [N]$ and $f(x) = x^{k+1}$.
It is expected that $sA-sA$ is not optimal, and we conjecture the following.

\begin{conj}\label{f conj}
Let $A$ be a set of reals and $f$ be a $k$-convex function. Then
\[ |2^kf(A-A) - (2^k-1)f(A-A)| \gg_k |A|^{k+1}. \]
\end{conj}

We feel that bounding the size of sumsets of convex functions with many summands as well as the iteration techniques to prove them are of independent interest. However, additionally they allow improvements of bounds for $2$-fold sumsets and energy of highly convex sets.
See \cite{BHRConvexity} for more details.

\subsection{Applications}
\subsubsection{Growth for products of generalised difference sets}
It was conjectured in \cite{BalogRoche-NewtonZhelezov} that for any $s >0$, there exists $m = m(s)$ such that if $A$ is a finite real set then 
\begin{equation}\label{growthA_A}
|(A-A)^{(m)}| \gg_s |A|^s,
\end{equation}
where $(A-A)^{(m)}$ denotes an $m$-fold product set $\underbrace{(A-A)\dots (A-A)}_{m \text{ times}}$.
This was proved in \cite{HansonRoche-NewtonZhelezov} for the case when $A \subset \Q$.
Balog, Roche-Newton and Zhelezov proved \eqref{growthA_A} for $s=3$, and additionally that for $s=17/8$, choosing $m = 3$ suffices; the first known results for $s>2$.
Recently, Hanson, Roche-Newton and Senger proved (implicitly) \cite{hansonRocheNewtonSenger2021} that \eqref{growthA_A} holds if $m = 8, s = 33/16$, but their method was stronger in the sense that some of the $A-A$ terms could be replaced with $A-a$ for specific values of $a\in A$. 
They use this to improve the best known lower bound for $|\Lambda(P)|$ where $\Lambda(P)$ is the set of dot products induced by a point set $P$ in $\R^2$. They proved
\[ |\Lambda(P)| \gtrsim |P|^{\frac{2}{3} + \frac{1}{3057}},  \]
the first result to break the threshold $|P|^{2/3}$.
See \cite{Rudnev_crossratios} for more connections between similar problems and growth in $A-A$.

We prove a result approaching this conjecture from a different direction, namely allowing for products of \emph{many}-fold differences.
\begin{thm}\label{Thm towards (A-A)^k}
Given any natural number $s \in \N$, there exists $m = m(s)$ such that if $A$ is a finite set of reals, then 
\[ |(sA-sA)^{(m)}| \gg_s |A|^{s}. \]
\end{thm}
The proof of Theorem \ref{Thm towards (A-A)^k} is an easy corollary of Theorem \ref{thm many A inside}, and is proved in \S\ref{manyinside section}.
If Conjecture \ref{f conj} holds, then \eqref{growthA_A} is the natural corollary.

\subsubsection{Sum-product type results}
Studying the sizes of sumsets and product sets which are not the traditional $A+A$ and $AA$ has led to many variations of the sum-product problem. See for example  \cite{BourgainChang2004, MurphyRudnevShkredovShteinikov,
StevensWarren2021}. 
What these and many other results \emph{do} share with the sum-product problem is they all enshrine the philosophy that additive structure and multiplicative structure cannot coexist in the same set. We may refer to any such result as a sum-product \emph{type} result. 

The sum-product problem has been studied in other fields as well. We particularly note that for $A\subset \C$, \eqref{sum-prod} is known for all $\delta < 1/3 + c$ (for some small c) \cite{basitLund} and for subsets of a function field $A \subset \F_q((t^{-1}))$, it is known for all $\delta < 1/5$ (with the implied constant $C$ also depending on $q$) \cite{BloomJones}. In this paper, we prove a related sum-product type result in each setting. Over the  complex numbers we have the following.

\begin{thm} \label{maintheoremC}
Let $A \subset \C$ be a finite set. Then the following holds:
\[ |A+A-A||AA|^2 \gg |A|^4. \]
\end{thm}

The function field $\F_q((t^{-1}))$ is the field of all Laurent series of the form \[ \sum_{i=-\infty}^k \alpha_i t^i, \qquad \text{where } \alpha_i \in \F_q \text{ for all } i. \]
We prove the following.
\begin{thm} \label{maintheoremFF}
For any finite $A \subset \F_q((t^{-1}))$ and any $\epsilon > 0$, we have
\[ |A+A-A|^3 |AA|^4 \gg_\epsilon q^{-2}|A|^{9-\epsilon}. \]
\end{thm}
Theorem \ref{maintheoremFF} does not extend to function fields where the base field is not finite. Furthermore, the dependence in this result on $q$ is necessary (and cannot be improved), since $\F_q((t^{-1}))$ has small non-trivial subfields. Indeed, setting $A=\F_q$ demonstrates this sharpness.
It is also worth mentioning that the same proof of Theorem \ref{maintheoremFF} holds for finite subsets of any field with nonarchimedean norm and finite residue field. In particular, it holds for finite subsets of the $p$-adic numbers $\Q_p$.

A form of Plunnecke's inequality \cite[Cor 1.5]{KatzShen_Plunnecke} shows that given any set $A$ in some group $G$, there exists $A' \subset A$ with $|A'|\geq |A|/2$ such that 
\[ |-A'+A+A| \ll \frac{|-A'+A|^2}{|A|}. \]
If $A' \subset A \in \C$, then applying Theorem \ref{maintheoremC} to $A'$ and some simple inequalities yields
\begin{equation}\label{plun_C}
|A-A|^2 |AA|^2 \gg |A|^5,
\end{equation}
which matches Solymosi's bound \cite{SolymosiComplexSP} (which has since been improved). Similarly, if $A' \subset A \in \F_q((t^{-1}))$, then applying Theorem \ref{maintheoremFF} to $A'$ yields 
\begin{equation}\label{plun_FF}
|A-A|^3 |AA|^2 \gg_\epsilon q^{-1}|A|^{6-\epsilon}.
\end{equation}
This matches Bloom and Jones' bound in \cite{BloomJones}, which is the best known sum-product bound in function fields with finite residue field.
Unfortunately, the bounds \eqref{plun_C} and \eqref{plun_FF} to not follow from our theorems if $A-A$ is replaced with $A+A$.

It is in the few products, many sums framework that Theorems \ref{maintheoremC} and \ref{maintheoremFF} are most relevant. In studying sum-product phenomena, it is often instructive to find bounds for $|A+A|$ given that $|AA|$ is small. In other words, given that there are \emph{few products}, we show that there are \emph{many sums}. This problem has been studied in \cite{BushCroot, Konyagin_h-fold,  MurphyRudnevShkredovShteinikov}.

In particular, we are addressing the few products, many $3$-fold sums problem. For example, if we know that $|AA| \approx |A|$, then for $A \subset \C$ we have
\[ |A+A-A| \gg |A|^{2} \]
and for $A \subset \F_q((t^{-1}))$ we have
\[ |A+A-A| \gg_{q,\epsilon} |A|^{5/3-\epsilon}. \]
The proofs of Theorems \ref{maintheoremC} and \ref{maintheoremFF} are both inspired by the combination of techniques used to prove Theorem \ref{maintheorem}.

\subsection{Structure of this paper}
In \S\ref{prelims section}, we present some essential preliminaries and a proof of Theorem \ref{maintheorem}, which forms the base step for the induction proof in the following section. 
In \S\ref{RemovingLogsSection}, we prove Theorem \ref{removing logs theorem} using a variation of the induction proof in \cite{higher_convexity}. 
\S\ref{manyinside section} is devoted to proving Theorem \ref{thm many A inside} through the auxiliary result Proposition \ref{thm many A inside induction}. We finally demonstrate that Theorem \ref{Thm towards (A-A)^k} is an easy corollary. 

The proofs of Theorems \ref{maintheoremC} and \ref{maintheoremFF} are almost identical to the corresponding proofs in \cite{SolymosiComplexSP} and \cite{BloomJones} respectively. Simple modifications using a version of Lemma \ref{mainlemma} produce the improvements. 
For this reason, both proofs appear only in Appendix \ref{FPMS section}, which is the only section in which we do not work exclusively over $\R$.

\section{Preliminaries}\label{prelims section}

\subsection{Convex functions and squeezing elements}
A simple and clever \emph{squeezing} approach for proving sumset bounds for convex sets was given in \cite{RuzsaShakanSolymosiSzemeredi}, providing a sharp lower bound
\[ |B+B-B| \gg |B|^2 \]
where $B$ is a convex set. This approach was significantly extended in \cite{higher_convexity}, providing estimates for longer sums of more convex sets and also the images $f(A)$ of sets with small additive doubling under $k$-convex functions.

The basic idea is the following.
Recall that a set $B = \{b_1 < \dots < b_N\}$ is convex if the adjacent differences between elements is increasing
\[ b_2-b_1 < b_3 - b_2 < \dots < b_N - b_{N-1}. \]
It follows that if $i$ is fixed, then for any $j < i$ the terms
\[ b_{i} + (b_{j}-b_{j-1}) \in B+B-B \]
are all unique and lie in $(b_i,b_{i+1})$.
We may apply this for any $i$ to obtain
$\sum_{i=1}^N (i-2) \approx |B|^2$
elements of $|B+B-B|$.

More can be said if we are looking at the images of convex functions. 
Given a function $f$, define its $d$-derivative by
\[ \Delta_d f(x) := f(x+d)-f(x). \]

\begin{lem}\label{convexity lemma}
If $f$ is a $k$-convex function, then for any $d$, $\Delta_d f$ is a $(k-1)$-convex function.
\end{lem}
\begin{proof}
We use induction on $k$. Suppose $f$ is $1$-convex.
We have 
\[ \Delta_d f(x) := f(x+d)-f(x) = \int_x^{x+d} f'(y)dy. \]
Since $f'$ is monotone, it follows that $\Delta_d f$ is also monotone, and hence $0$-convex. 

Next assume the statement holds for $(k-1)$-convex functions. 
Let $f$ be a $k$-convex function.
By definition, this implies that $f'$ is a $(k-1)$-convex function.
The induction hypothesis implies that $\Delta_d(f')$ is a $(k-2)$-convex function. But since $\Delta_d(f') = (\Delta_d f)'$, it follows that $\Delta_d f$ is $(k-1)$-convex, completing the induction. 
\end{proof}

If $f$ is convex, $A = \{a_1 < \dots < a_N\}$ and $d >0$, then Lemma \ref{convexity lemma} implies that $\Delta_d f$ is increasing, meaning that
\[ f(a_1 + d) - f(a_1) < \dots < f(a_N + d) - f(a_N).\]
Consequently if $i$ is fixed, then for any $j<i$ the terms
\[f(a_i) + f(a_j+d) - f(a_j)\]
are all different and lie in $(f(a_i),f(a_i + d))$.
We usually take $d$ to be at most the smallest difference between adjacent elements of $A$ to ensure that these intervals are disjoint.
These observations are summarised in the following.

\begin{lem}[The squeezing lemma]\label{lemma squeeze}
Let $f$ be a convex function and $d > 0$. Let $s$ be a real number and $S_-$ a set of real numbers smaller than $s$. Then
\[ f(s) + \Delta_d f(S_-) \subset (f(s),f(s+d)). \ \]
\end{lem}

\subsection{The basic argument}
The following method will be extended to different fields ($\C$ and $\F_q((t^{-1}))$) in Appendix \ref{FPMS section}, but ideas herein will also be employed in $\S$\ref{RemovingLogsSection}.

We will also introduce the following notation: if $a'<a$, then
\[ n_A(a',a) := (A+A-A)\cap (a',a]. \]
Taking $a,a' \in A$, the quantity $n_A(a',a)$ tells us about how $A+A-A$ is distributed among the elements of $A$.  
Intuitively we would think that wide intervals contain many elements of $A+A-A$. This intuition is quantified by the following Lemma, which also appears in \cite{BHRConvexity}. 

\begin{lem}\label{mainlemma}
Let $D:=\{d_1<d_2\dots<d_{|D|}\}$ be the positive differences in $A-A$. If $a,a' \in A$ with $a'<a$ and $n_A(a',a)\leq Z$, then $a-a' \leq d_Z$.
\end{lem}

In other words, if there are at most $Z$ elements of $A+A-A$ in $(a',a]$, then $a-a'$ must be among the $Z$ smallest positive differences in $A-A$.

\begin{proof}
If not then $a-a' = d_Y$ where $Y>Z$. But then
\[ a'<a'+d_i \leq a, \]
for $i = 1,\dots,Y$. Thus there are at least $Y>Z$ elements of $A+A-A$ in $(a',a]$, contradicting that $n_A(a',a) \leq Z$.
\end{proof}

We are now equipped to prove Theorem $\ref{maintheorem}$.
\begin{proof}[Proof of Theorem \ref{maintheorem}]
Let $A := \{a_1 < \dots < a_{|A|}\}$.
We say that $a_i$ is \emph{good} if 
\[ n_A(a_i,a_{i+1}) \ll \frac{|A+A-A|}{|A|} \qquad \text{and} \qquad n_{f(A)}(f(a_i),f(a_{i+1})) \ll \frac{|f(A) + f(A) - f(A)|}{|A|}. \]
Since 
\[\sum_{i=1}^{|A|-1} n_A(a_i,a_{i+1}) \leq |A+A-A| \quad \text{and} \quad \sum_{i=1}^{|A|-1} n_{f(A)}(f(a_i),f(a_{i+1})) \leq |f(A)+f(A)-f(A)|,\]
by the pigeonhole principle, there is a set $A'$ with $|A'| \gg |A|$ such that each element of $A'$ is good.

Now consider the map
\[ \Psi: a_i \mapsto (a_{i+1}-a_i, f(a_{i+1}) - f(a_i)). \]
By the mean value theorem, there exits a sequence $\{c_i\}$ where $c_i \in (a_i,a_{i+1})$
and 
\[ \frac{f(a_{i+1})-f(a_i)}{a_{i+1}-a_i} = f'(c_i). \]
Once $f(a_{i+1})-f(a_i)$ and $a_{i+1}-a_i$ are fixed, $c_i$ is known uniquely since $f'$ is strictly monotone. Thus $a_i$ is also known uniquely, and $\Psi$ is injective.

Restrict the domain of $\Psi$ to $A'$.  
Since $\Psi$ is injective, the size of its domain $A'$ equals the size of its image $\Psi(A')$, proving that
\begin{equation}\label{image equation}
|A| \ll |A'| = |\Psi(A')|.
\end{equation}
A suitable upper bound for $|\Psi(A')|$ will complete the proof. 

Since each $a_i \in A'$ is good it satisfies 
\[ n_{A}(a_i,a_{i+1}) \ll \frac{|A+A-A|}{|A|}. \]
Lemma \ref{mainlemma} shows that $a_{i+1}- a_i$ is among the smallest $\frac{|A+A-A|}{|A|}$ positive elements in $A-A$ and therefore, there are $\ll \frac{|A+A-A|}{|A|}$ values it can take. By an identical argument there are $\ll \frac{|f(A)+f(A)-f(A)|}{|A|}$ values $f(a_{i+1})-f(a_i)$ can take.

This proves that $|\Psi(A')| \ll |A+A-A||f(A)+f(A)-f(A)||A|^{-2}$.
It follows from \eqref{image equation} that
\[ |A| \ll |A+A-A||f(A)+f(A)-f(A)||A|^{-2}, \]
and rearranging completes the proof.
\end{proof}

\begin{rmk}
In \S\ref{RemovingLogsSection}, we will claim that Theorem \ref{maintheorem} is the base case for the induction proof of Theorem \ref{removing logs theorem}.
In fact, we will need the slightly stronger statement that $f(A)+f(A)-f(A)$ contains $\gg |A|^3 |A+A-A|^{-1}$ elements in $(\min(f(A)),\max(f(A)))$. 

To see that this is still true, note that given any triple $(x_1,x_2,x_3) \in f(A)^3$ with $x_1\leq x_2 \leq x_3$, we have $x_1 + x_3 - x_2 \in [\min f(A),\max f(A)]$. Thus for large $|A|$, a little less than one third of $f(A)+f(A)-f(A)$ must lie in $(\min f(A),\max f(A))$.
\end{rmk}

\section{Proof of Theorem \ref{removing logs theorem}}\label{RemovingLogsSection}

\begin{proof}[Proof of Theorem \ref{removing logs theorem}]
The proof will be by induction on $k$. 
The actual statement we will prove is the slightly stronger statement that all sums produced lie in the interval $(\min f(A),\max f(A))$.
So the base step (which is proved in Theorem \ref{maintheorem}) states that $f(A)+f(A)-f(A)$ contains $\gg |A|^3 |A+A-A|^{-1}$ elements in $(\min f(A),\max f(A))$.

As in previous proofs, we say that $a_i\in A$ is \emph{good} if 
\[ n_A(a_i,a_{i+1}) \ll \frac{|A+A-A|}{|A|}. \]
By the pigeonhole principle, a positive proportion of all $a_i \in A$ are good. We henceforth restrict our attention to the good $a_i$.
Consider the differences $a_{i+1}-a_i$ and let $H$ be the set of all such differences.
Since we are only considering good values of $a_i$, Lemma \ref{mainlemma} implies that $|H| \ll |A+A-A||A|^{-1}$. For each $h \in H$, define 
\[ A_{h} = \{a_i: a_{i+1}-a_i = h\}. \]
We furthermore know that $\sum_{h\in H} |A_{h}| \approx |A|$.

If $A_{h} = \{ a_{e_1}< \dots <a_{e_L}\}$
then let $A_h^i = \{ a_{e_1}< \dots <a_{e_{i-1}}\}$ be the truncation taking only the smallest $i-1$ elements of $A_h$.
For any $a_{e_i} \in A_{h}$, the squeezing lemma (Lemma \ref{lemma squeeze}) implies that
\[ f(a_{e_i}) + \Delta_{h}f(A_h^i) \subset (f(a_{e_i}),f(a_{e_i+1})). \]
Since $f$ is a $k$-convex function, Lemma \ref{convexity lemma} proves that $g_i := f(a_{e_i}) + \Delta_{h}f$ is a $(k-1)$-convex function, and we have
\[ g_i(A_{h}^i) \subset (f(a_{e_i}),f(a_{e_i+1})). \]

It follows from the induction hypothesis that 
\[ 2^{k-1}g(A_{h}^i) - (2^{k-1}-1)g(A_{h}^i) \subset 2^{k}f(A) - (2^{k}-1)f(A) \]
contains $\gg \frac{|A_{h}^i|^{2^{k}-1}}{|A_{h}^i+A_{h}^i-A_{h}^i|^{2^{k}-k-1}}$ elements in $(f(a_{e_i}),f(a_{e_i+1}))$.
This argument can be run for every element of $A_{h}$, and then also for each $h \in H$ to obtain
\begin{align}
    |2^k f(A) - (2^{k-1}-1) f(A)| &\gg \sum_{h \in H} \sum_{i=1}^{|A_h|} \frac{|A_{h}^i|^{2^{k}-1}}{|A_{h}^i+A_{h}^i-A_{h}^i|^{2^{k}-k-1}} \nonumber\\
    & \gg \frac{1}{|A+A-A|^{2^{k}-k-1}}\cdot \sum_{h \in H} |A_{h}|^{2^{k}}. \label{big approx removing logs}
\end{align}
Above we have used the trivial facts that $|A+A-A| \gg |A_{h}^i+A_{h}^i - A_{h}^i|$ and 
$|A_h^i| = i-1$.
Now by H\"older's inequality, we have 
\begin{equation}\label{Holders removing logs}
\sum_{h \in H} |A_{h}|^{2^k} \cdot |H|^{2^k-1} \geq \left(\sum_{h\in H} |A_{h}|\right)^{2^k} \approx |A|^{2^k}.
\end{equation}
Recalling that $|H| \ll |A+A-A||A|^{-1}$, \eqref{big approx removing logs} and \eqref{Holders removing logs} yield the desired
\[ |2^k f(A) - (2^{k-1}-1) f(A)| \gg  \frac{|A|^{2^{k+1}-1}}{|A+A-A|^{2^{k+1}-k-2}}, \]
with all constructed elements lying in $(\min f(A),\max f(A))$.
\end{proof}

\begin{rmk}
In \cite{higher_convexity}, Lemma 3.1 is used to find a consecutive difference $h\in A-A$ with many realisations, and this is done by dyadic pigeonholing. 

Instead we have found a large set of consecutive pairs $(a_i,a_{i+1})$ which don't have many elements of $A+A-A$ in between them. By Lemma \ref{mainlemma}, there are few possible values that $a_{i+1}-a_i$ can take, and therefore the pigeonhole principle proves that some of these differences must be reaslised many times. 
This approach avoids the logarithmic factors intrinsic to a dyadic pigeonholing argument. 
\end{rmk}

\section{Proof of Theorem \ref{thm many A inside}}\label{manyinside section}

In proving Theorem \ref{thm many A inside}, we will actually prove the following stronger but more cumbersome result, because it makes the induction step more manageable.

\begin{prop}\label{thm many A inside induction}
Let $A = \{a_1 < \dots < a_N\}$ be any set or reals, $k$ be a positive integer and $f$ be a $k$-convex function. Also let $d >0$ be such that
\begin{equation}\label{d condition}
kd \leq a_i-a_j \qquad \text{for all} \qquad j<i.
\end{equation}
For $i = 1,\dots, N$, set 
\[ P_{k,i} = \{ a_i, a_i + d, \dots, a_i + kd \} \qquad \text{and} \qquad S_{k,i} = \cup_{j=1}^{i-1} P_{k,j}. \]
Then the set
\[ 2^kf(S_{k,N}) -(2^k-1)f(S_{k,N}) \]
contains $\gg_k |A|^{k+1}$ elements in $(\min f(A),\max f(A))$.
\end{prop}

\begin{proof}\label{thmA-A}
The proof is by induction on $k$. 
We begin with the base step. Let $d < a_i-a_j$ for all $j<i$. We have
\[ S_{1,N} = \{ a_1, a_1+d, a_2, a_2 + d, \dots, a_{N-1}, a_{N-1} +d \}. \]
Since $f$ is a convex function it follows that 
\[ f(a_1+d)-f(a_1) < \dots < f(a_{N-1}+d) - f(a_{N-1}) \]
and consequently if $i$ is fixed, then for any $j<i$ the terms
\[f(a_i) + f(a_j+d) - f(a_j)\]
are all different and lie in the interval $I_i = (f(a_i),f(a_i + d))$. This produces $i-1$ elements of $f(S_1) + f(S_1) - f(S_1)$. 
Since $d < a_i-a_j$ for all $j<i$,
the intervals $I_i$ are disjoint.
Apply this argument for $i = 1, \dots , N-1$, producing
\[ \sum_{i=1}^{N-1} (i-1) \approx |A|^2 \]
elements of $f(S_1) + f(S_1) - f(S_1)$ lying in $(\min f(A), \max f(A))$, thus completing the base step.

We proceed to the induction. 
Note that \eqref{d condition} guarantees that the intervals
\[ (a_i, a_i+kd) \qquad 1\leq i \leq N \]
spanned by the sets $P_{k,i}$ are all disjoint.
Using the squeezing lemma (Lemma \ref{lemma squeeze}) and \eqref{d condition}, we have 
\[ f(a_i) + \Delta_d f(S_{k-1,i}) \subset (f(a_i), f(a_{i}+d)) \subset (f(a_i),f(a_{i+1})). \]
Since $f$ is $k$-convex, the new functions $g_i:= f(a_i) + \Delta_d f$ are all $(k-1)$-convex by Lemma \ref{convexity lemma}, and also
\[ g_i(S_{k-1,i}) \subset (f(a_i),f(a_{i+1})). \]
From \eqref{d condition}, the inequality $(k-1)d \leq a_i-a_j$ trivially holds for all $j<i$, so the induction hypothesis can be applied to $A_i := \{a_1, \dots , a_{i}\}$, showing that the set
\[ 2^{k-1}g_i(S_{k-1,i}) - (2^{k-1}-1)g_i(S_{k-1,i}) \subset 2^k f(S_{k,N}) - (2^k-1) f(S_{k,N}) \]
contains $\gg |A_i|^{k} = (i-1)^{k}$ elements in $(f(a_i),f(a_{i+1}))$.
Applying this for each function $g_i$ we get
\[ |2^k f(S_{k,N}) - (2^k-1) f(S_{k,N})| \gg \sum_{i=1}^{N} (i-1)^k \approx_k |A|^{k+1}, \]
and all constructed elements lie in $(\min f(A),\max f(A))$, closing the induction.
\end{proof}

\begin{proof}[Proof of Theorem \ref{thm many A inside}]
Given $A = \{a_1 < \dots < a_N\}$, let $b,b' \in A$ be such that $d_0 = b-b'$ is the smallest positive element of $A-A$. 

We start by proving the case where $k= 2s$ is even. 
We set  
\[ A' = \{ a_k - sd_0, \dots, a_{kM}-sd_0 \}, \]
where $M = \lfloor N/k \rfloor$.
Now apply Proposition \ref{thm many A inside induction} to $A'$ with $d = d_0$.
Since $d_0 \in A-A$ we get $S_{k,M} \subset (s+1)A-sA$, and the result follows.

If $k = 2s-1$ is odd, then instead using 
\[ A' = \{ a_k - sd_0 -b', \dots, a_{kM}-sd_0-b' \} \]
completes the proof.
\end{proof}

We now prove Theorem \ref{Thm towards (A-A)^k}.
\begin{proof}[Proof of Theorem \ref{Thm towards (A-A)^k}]
Set $f(x) = \log(x)$ in the case $k = 2s-1$ of Theorem \ref{thm many A inside}. Using a crude upper bound on the size of the quotient set, we get that for any natural number $s$, 
\[ |(sA-sA)^{(2^k)}|^2 \gg \left|\frac{(sA-sA)^{(2^k)}}{(sA-sA)^{(2^k-1)}}\right| \gg |A|^{2s}. \]
Taking square roots and setting $m(s) = 2^k = 2^{2s-1}$ completes the proof.
\end{proof}

One can think of a $k$-convex set $A$ as $f([N])$ where $f$ is a $k$-convex function. Then \cite[Theorem 1.3]{higher_convexity} shows that 
\[ |2^kA - (2^k-1)A| \gg_k |A|^{k+1}. \]
This is also a corollary of Proposition \ref{thm many A inside induction} by setting $A$ to be an arithmetic progression. In fact, Proposition \ref{thm many A inside induction} is designed to synthesise the important properties of $[N]$ which make $f([N])$ grow under many sums and differences.

\section*{Acknowledgements}
I would like to thank Misha Rudnev for useful input and suggesting some of these problems. 
I would also like to thank Tom Bloom, Oleksiy Klurman, Sam Mansfield, Akshat Mudgal, Jonathan Passant and Oliver Roche-Newton for useful conversations.

\bibliographystyle{plain}
\bibliography{bibliography}

\appendix

\section{Proof of Theorems \ref{maintheoremC} and \ref{maintheoremFF}}\label{FPMS section}

Let $(\F,\|\cdot\|)$ be a field with a norm.
In this section, we will use extensively the notation
\[ B(a,r) = \{ x \in \F: \|x-a\| \leq r \}. \]
That is, $B(a,r)$ is the ball around $a$ with radius $r$.

We will need a version of Lemma \ref{mainlemma} which is applicable in $\C$ and in $\F_q((t^{-1}))$.

\begin{lem}\label{mainlemmaOtherF}
Let $(\F,\|\cdot\|)$ be a field with a norm, and let $A$ be a finite subset of $\F$.
Write $D := A-A = \{0\} \cup \{ d_1,\dots, d_{|D|} \}$ such that the norms of the $d_i$ are non-decreasing.
If $a,a' \in A$ and
\[ |(A+A-A)\cap B(a,\|a-a'\|)| \leq Z, \]
then $\|a-a'\| \leq d_Z$.
\end{lem}

The proof is almost identical to the proof of Lemma \ref{mainlemma}.

\begin{proof}
If not then $\|a-a'\| = d_Y$ where $Y>Z$, whence
\[ a+d_i \in B(a,\|a-a'\|), \]
for $i = 1,\dots,Y$. 
This produces $Y>Z$ elements of $A+A-A$ in $B(a,\|a-a'\|)$, a contradiction.
\end{proof}

Importantly Lemma \ref{mainlemmaOtherF} can be applied to finite subsets of $\C$ with the usual complex modulus as the norm, and $\F_q((t^{-1}))$ with $\|x\| = q^{\deg x}$ as the norm.

\subsection{Sum-product type result in $\C$}
We prove Theorem \ref{maintheoremC} following a method of Solymosi \cite{SolymosiComplexSP} incorporate Lemma \ref{mainlemmaOtherF} to obtain an improvement.
\begin{proof}[Proof of Theorem \ref{maintheoremC}]
For notation, let $B_A(a):= B(a,|a-b|)$ where $b$ is a nearest neighbour of $a$ (according to the standard modulus function in $\C$).
We will say that $(a,b,c) \in A^3$ is \emph{good} if:
\begin{align}
    &b \in B_A(a)\backslash \{a\}, \nonumber \\ 
    &|(A+A-A)\cap B_A(a)| \ll \frac{|A+A-A|}{|A|}, \text{ and } \label{CGCond+}\\
    &|(AA)\cap c \cdot B_A(a)| \ll \frac{|AA|}{|A|}. \label{CGCond*}
\end{align}

Note that balls of the form $B_A(a)$ have no elements of $A\backslash\{a\}$ in their interior. It can be easily shown that no complex number can be contained in more than $7$ such balls. 
Therefore, we have 
\begin{equation}\label{CGood1}
\sum_{a\in A} |(A+A-A)\cap B_A(a)| = \sum_{v \in A+A-A}\sum_{a\in A} \mathds{1}_{v\in B_A(a)} \leq 7|A+A-A|,
\end{equation}
and for any $c \in A$
\begin{equation}\label{CGood2}
\sum_{a\in A} |(AA)\cap (B_A(a))| = \sum_{v \in AA/c}\sum_{a\in A} \mathds{1}_{v\in B_A(a)} \leq 7|AA|.
\end{equation}
Since each $a \in A$ has at least one nearest neighbour, applying the pigeonhole principle to \eqref{CGood1} and \eqref{CGood2}, there exists a subset $T \in A^3$ with $|T| \gg |A|^2$ such that each triple $(a,b,c)\in T$ is good.
Now consider the set $T$ under the map
\[ \Psi: (a,b,c) \mapsto (a-b,ca,cb). \]
As long as $a\neq b$ (which certainly holds for all triples in $T$), $\Psi$ is injective, whereupon
\[ |A|^2 \ll |T| = |\Psi(T)|. \]

We now search for an upper bound on $|\Psi(T)|$ and will do this crudely by counting, given that $(a,b,c)$ is a good triple, how many values $a-b$ may take, and then separately, how many values the pair $(ac,bc)$ can take.

Because $(a,b,c)$ is good, $a$ satisfies \eqref{CGCond+}. Also, since $b \in B_A(a)$, we know that $B_A(a) = B(a,|a-b|)$, so Lemma \ref{mainlemmaOtherF} proves that there are $\ll  |A+A-A||A|^{-1}$ possible values that $a-b$ may take.

The number of values of that $ac$ may take is trivially upper-bounded by $|AA|$. Once $ac$ is fixed, since $bc$ lies in $c\cdot B_A(a)$ and \eqref{CGCond*} holds, there are $\ll|AA||A|^{-1}$ values that $bc$ may take.

Putting this together we obtain 
\[ |A|^2 \ll |A+A-A||AA|^2|A|^{-2}, \]
which we rearrange to arrive at the desired result.
\end{proof}

\subsection{Sum-product type result in $\F_q((t^{-1}))$}
Next we prove Theorem \ref{maintheoremFF} with the method of Bloom and Jones \cite{BloomJones} with necessary modifications to incorporate the $|A+A-A|$ term. 

We will firstly introduce some notation that will be used throughout the proof. We can put a norm structure on $\F_q((t^{-1}))$ by saying that $\|x\| = q^{\deg x}$, where $\deg$ is the standard degree for Laurent series. Observe that the norm of a difference $\|a-a'\|$ is a measure of how similar $a$ and $a'$ are; that is, to the right of which term the Laurent series of $a$ and $a'$ agree. 
Next let
\[ r_A(a) = \min_{a'\in A\backslash \{a\}} \|a-a'\|, \]
and 
\[ B_A(a) = B(a,r_A(a)). \]

We will argue that any intersecting balls in $\F_q((t^{-1}))$ are nested; that is, if $y \in B(x_1,r_1) \cap B(x_2,r_2)$, then either $B(x_1,r_1) \subset B(x_2,r_2)$ or $B(x_2,r_2) \subset B(x_1,r_1)$. Indeed, if $r_1 \leq r_2$ the former occurs, if $r_2 \leq r_1$ the latter occurs. It follows that if $r_1 = r_2$ then $B(x_1,r_1) = B(x_2,r_2)$.

For a proof suppose $r_1 \leq r_2$ and that $y \in B(x_1, r_1) \cap B(x_2,r_2)$. This means that the Laurent series for $x_1, y$ agree on degrees $\geq r_1$ and $x_2, y$ agree on degrees $\geq r_2$. Since $r_1\leq r_2$, this implies that $x_1,x_2$ agree on degrees $\geq r_2$.
Now if $w \in B(x_1,r_1)$, then $w$ agrees with $x_1$ on degree $\geq r_1$. It follows that $w$ agrees with $x_2$ on degree $\geq r_2$. This demonstrates that $w \in B(x_2,r_2)$, completing the proof. 

We will use the same definition and notation for separable sets and $A$-chains as in \cite{BloomJones}.

\begin{define}
A finite set $A \in \F_q((t^{-1}))$ is \emph{separable} if its elements can be indexed as 
\[ A = \{ a_1,\dots, a_{|A|} \} \]
such that for each $1\leq j \leq |A|$ there is a ball $B_j$ with 
\[ A\cap B_j = \{ a_1,\dots, a_j \}. \]
We also say that $\mathcal{C} = (c_1, \dots, c_n) \in A^n$ is an \emph{$A$-chain} of length $n$ if all the $c_i$ are different and
\[ B_A(c_1) \subset \dots \subset B_A(c_n). \]
\end{define}

The following two lemmas are proved in \cite{BloomJones}. We list them here without proof.
\begin{lem}\label{lem1}
If $A \subset \F_q((t^{-1}))$ is a separable set, then for any natural numbers $k,n,m$ such that $n+m = k$,
\[ |nA-mA| \gg_k |A|^k. \]
\end{lem}
\begin{rmk}
In \cite{BloomJones}, this Lemma is stated only for sumsets, not difference sets. However, since the result is proved by showing that the corresponding energy is minimum possible, it also proves the corresponding bound for difference sets. 
\end{rmk}

\begin{lem}\label{lem2}
If the elements of $\mathcal{C}$ form an $A$-chain, then $\mathcal{C}$ contains a separable set of size at most $|\mathcal{C}|/q$.
\end{lem}

The strategy of our proof is as follows:
if we can find a suitably large $A$-chain, then Lemma \ref{lem1} shows that it contains a large separable set $U$. Then Lemma \ref{lem2} shows that $U$ has a large $k$-fold sumset, and therefore so does $A$. Applying Pl\"{u}nnecke's Inequality will then complete the proof. Thus the key result is the following:

\begin{prop}\label{mainprop}
Let $A \subset \F_q((t^{-1}))$ be finite. Then $A$ contains an $A$-chain $\mathcal{C}$ with 
\[ |\mathcal{C}| \gg \frac{|A|^4}{|A+A-A||AA|^2(\log|A|)^3}. \]
\end{prop}

\begin{proof}
For each $a \in A$ write $N(a)$ to be the length of the longest $A$-chain $(c_1,\dots,c_k)$ where $c_k = a$. We begin by dyadic pigeonholing: for each $0 \leq j \leq \log |A|$, let $A_j$ be the set of $a\in A$ such that $2^j \leq N(a) < 2^{j+1}$. There exists some $j_0$ such that $|A_{j_0}| \geq |A|/\log|A|$.

To complete the proof, it suffices to show that
\[ 2^{j_0} \gg \frac{|A|^4}{|A+A-A||AA|^2(\log|A|)^3}. \]
To this end, we say that a triple $(a,b,c)\in A^3$ is \emph{good} if
\begin{align}
    &a \in A_{j_0}, \label{FFCondA}\\
    &b \in B_A(a)\backslash \{a\}, \label{FFCondb}\\
    &|(A+A-A)\cap B_A(a)| \ll \frac{2^{j_0}|A+A-A|}{|A_{j_0}|}, \text{ and} \label{FFGCond+}\\
    &|(AA)\cap c \cdot B_A(a)| \ll \frac{2^{j_0}|AA|}{|A_{j_0}|}. \label{FFGCond*}
\end{align}
Let $T$ be the set of all good triples.
We complete the proof by showing the following upper and lower bounds on $|T|$:
\begin{align}
    |T| &\gg 2^{j_0}|A_{j_0}||A| \label{lower} \\
    |T| &\ll \frac{2^{2j_0}|A+A-A||AA|^2}{|A_{j_0}|^2}. \label{upper}
\end{align}

We begin by proving \eqref{lower}. Observe that
\begin{equation}\label{pigeonholeStepFF}
\sum_{a\in A_{j_0}} |(A+A-A)\cap B_A(a)| = \sum_{v \in A+A-A \ }\sum_{a\in A_{j_0}} \mathds{1}_{v\in B_A(a)} = \sum_{v \in A+A-A} |C_{j_0}(v)|,
\end{equation}
where $C_{j_0}(v)$ is the set of $a \in A_{j_0}$ such that $v \in B_A(a)$. 
Similarly, for any $c \in A$
\begin{equation}\label{pigeonholeStep*FF}
\sum_{a\in A_{j_0}} |(AA)\cap c\cdot B_A(a)| = \sum_{u \in AA \ }\sum_{a\in A_{j_0}} \mathds{1}_{v\in c\cdot B_A(a)} = \sum_{v \in AA/c \ }\sum_{a\in A_{j_0}} \mathds{1}_{v\in B_A(a)} = \sum_{v \in AA/c} |C_{j_0}(v)|.
\end{equation}

It is worth noting that unike when we are working in $\C$, $|C_{j_0}(v)|$ is not bounded by any constant. However, for all $a \in C_{j_0}(v)$, the corresponding balls $B_A(a)$ all share the point $v$ and are therefore nested. In other words, the elements of $C_{j_0}(v)$ can be ordered to form an $A$-chain. It follows that 
\begin{equation*}\label{N(a) condition}
|C_{j_0}(v)| \leq N(a) \leq 2^{j_0+1}.
\end{equation*}
Applying the pigeonhole principle to \eqref{pigeonholeStepFF} and \eqref{pigeonholeStep*FF}, there is a subset $A' \subset A_{j_0}$ with $|A'| \gg |A_{j_0}|$ such that for each $a \in A'$ and $c \in A$, \eqref{lower} and \eqref{upper} both hold. 

Given an $A$-chain $\mathcal{C} = (c_1,\dots,c_{N(a)})$ with $c_{N(a)} = a$, the definition of an $A$-chain shows that $c_i \in B_A(a)$ for $i = 1, \dots, N(a)$. It follows that
\begin{equation}\label{b bound}
2^{j_0} \leq N(a) \leq |B_A(a)\cap A|.
\end{equation}
Now for any $c \in A$ and $a \in A_{j_0}$, it follows that \eqref{FFCondA},\eqref{FFGCond+},\eqref{FFGCond*} all hold. Once $a$ is fixed \eqref{b bound} shows that at least $2^{j_0}$ values of $b$ will satisfy \eqref{FFCondb}, completing the proof that
\[|T| \gg 2^{j_0}|A_{j_0}||A|.\]

We now prove \eqref{upper}. The map
\[ \Psi:(a,b,c) \mapsto (a-b,ac,bc) \]
is manifestly injective when restricted to $T$. Similar to the proof of Theorem \ref{maintheoremC}, we upper bound $|\Psi(T)|$ by
\[ \frac{2^{2j_0}|A+A-A||AA|^2}{|A_{j_0}|^2}. \]

Since $a$ satisfies \eqref{FFGCond+} and $b \in B_A(a)$, Lemma \ref{mainlemmaOtherF} proves that there are $\ll  \frac{2^{j_0}|A+A-A|}{|A|_{j_0}}$ possible values that $a-b$ may take.
The number of values of that $ac$ may take is trivially upper-bounded by $|AA|$. Once $ac$ is fixed, since $bc$ lies in $c\cdot B_A(a)$ and \eqref{FFGCond*} holds, there are $\ll\frac{2^{j_0}|AA|}{|A_{j_0}|}$ values that $bc$ may take.

Putting this all together, we get 
\[|T| \ll \frac{2^{2j_0}|A+A-A||AA|^2}{|A_{j_0}|^2}\]
whereupon using $|A_{j_0}| \gg \frac{|A|}{\log |A|}$, and rearranging, completes the proof.
\end{proof}

\begin{proof}[Proof of Theorem \ref{maintheoremFF}]
By Lemma \ref{lem2} and Proposition \ref{mainprop}, there exists a separable subset $U \in A$ of size at least
\[H:=\frac{|A|^4}{q|A+A-A||AA|^2(\log|A|)^3}. \]
Then using Lemma \ref{lem1} and Plunnecke's inequality (see \cite{Petridis_Plunnecke} for a short proof), we have
\[ \frac{|A+A-A|^k}{|A|^{k-1}} \gg |kA-kA| \gg |kU-kU| \gg H^{2k}.\]
Taking $k$th roots we get
\[ |A+A-A| \gtrsim H^2 |A|^{1-1/k} = \frac{|A|^{9-1/k}}{q^2|A+A-A|^2|AA|^4}. \]
Rearranging yields the desired result for sufficiently large $k$.
\end{proof}

\end{document}